\documentclass[12pt]{amsart}

\usepackage{graphicx}

\usepackage{float}
\usepackage{subfig}
\usepackage{caption}
\usepackage[nobysame, alphabetic]{amsrefs}

\newtheorem{thm}{Theorem}[section]
\newtheorem{defn}[thm]{Definition}

\newtheorem{prop}[thm]{Proposition}
\newtheorem{lem}[thm]{Lemma}
\theoremstyle{definition}
\newtheorem{rmk}[thm]{Remark}

\def\HH{\mathbb{H}}
\def\dHH{\partial \mathbb{H}}

\def\MM{\Sigma_{1,1}}
\def\tMM{\widetilde\Sigma_{1,1}}
\def\MMm{\Sigma}

\def\ZZ{\mathbb{Z}}

\def\RR{\mathbb{R}}

\def\isomH{\text{isom}^+(\HH)}
\def \moduli{\mathcal{M}(\Sigma)}
\def \teich{\mathcal{T}(\Sigma)}

\title{On systoles and ortho spectrum rigidity}
\author[Masai]{Hidetoshi Masai}
\address{Department of Mathematics Tokyo Institute of Technology 2-12-1, Ookayama, Meguro-ku, Tokyo. 152-8551. Japan}
\email{masai at math.titech.ac.jp}

\author[McShane]{Greg McShane}
\address{Institut Fourier 100 rue des maths, BP 74, 38402 St Martin d'H\`eres cedex, France}
\email{mcshane at univ-grenoble-alpes.fr}

\begin{document}
\maketitle

\begin{abstract}
We consider the ortho spectrum of hyperbolic surfaces with totally geodesic boundary.
We show that in general the ortho spectrum does not 
determine the systolic length 
but that there are only finitely many possibilities.
%for the the systolic length  for a given ortho spectrum.
As a corollary we show that,  up to isometry,
there are only finitely many 
hyperbolic structures on a surface 
that share a  given ortho spectrum.
\end{abstract}

\section{Introduction}

\subsection{Context}
Mark Kac famously asked if one could hear  the shape of a drum (c.f. \cite{Gordon}).
This has a more precise formulation  namely: 
does the spectrum of the Laplacian of a closed hyperbolic surface determine
the metric up to isometry?
Huber and Selberg (c.f. \cite{Buser}) had shown that for such surfaces
the Laplacian's spectrum is determined by the \textit{length spectrum}
which is the set of  lengths of the closed geodesics counted with multiplicities.
Examples were found of   \textit{isospectral pairs}
that is hyperbolic structures $X,Y$
 on a closed surface $\MMm$
which have the same length spectrum
but are  not isometric 
and so Kac's question was answered in the negative.
In particular, there is a very natural construction using finite covers  
due to Sunada (\cite{Sunada}, see also \cite{Buser}) which allows one to construct many such pairs of surfaces.

Kac's original question about analytic data associated to a metric (the Laplacian)
was translated to a question about metric data (the length of geodesics)
and solved, at least partially, using what is essentially algebra.
We say ``partially" because it seems reasonable to conjecture
that any pair of isospectral surfaces
are \textit{commensurable}
that is they share a common finite cover.

Motivated by Kac's question, 
we consider an analogous problem.
Basmajian \cite{Bas} defined the \textit{ortho spectrum}
of a %convex 
hyperbolic structure $X$ on a surface $\MMm$ of finite type with non empty  (totally geodesic) boundary.
He proved a formula for the \textit{perimeter},
 which we define to be the sum of the lengths of the boundary components,
 in terms of the ortho spectrum. Subsequently, Bridgeman and Kahn \cite{Bridgeman, BK} gave a formula for the area of the surface in terms of the ortho spectrum. 
Thus two of the principal moduli of the hyperbolic structure $X$ are explicitly determined by the ortho spectrum.

It turns out that when one
replaces  length spectrum  by ortho spectrum,
the resulting analogue of Kac's question
is easier  to answer in the negative.
This is because it is far easier to construct 
the required isospectral pairs for the ortho spectrum
by   using finite covers
than for the  length spectrum.
So we set out to understand to what extent the 
ortho spectrum determines the geometry 
of the hyperbolic structure $X$
knowing that  
three important numerical invariants (moduli) of $X$
are determined  by the ortho spectrum:
\begin{itemize}
\item The total length of the boundary by Basmajian \cite{Bas}.
\item The area of the surface by Bridgeman \cite{Bridgeman}.
\item The entropy of the geodesic flow  can be determined via Poincar\'e series by J. Parker \cite{Par}.
\end{itemize}
Our point of view is much influenced by the work of Buser and his students who have studied Kac's original 
question in an attempt to give a complete answer.
In fact Buser and Semmler \cite{BS} proved that 
for a one holed torus the spectrum does determine 
a hyperbolic structure up to isometry.
Of fundamental importance in Buser's work is
the  \textit{systolic length} which does not
appear on our list of moduli 
explicitly determined by the ortho spectrum.
Recall that the systolic length 
is the length of the  shortest closed geodesic
and it is easy to see that this is the infimum of the length spectrum.
The systolic length is 
obviously a measure  of how thin a surface gets 
but the Collar Lemma tells us that a short geodesic 
lives in wide collar
so it also gives a lower bound on the diameter of a surface.

\subsection{Statement of Results}
Though the relationship between systolic length and 
length spectrum is quite straightforward
its relationship to the ortho spectrum is less obvious
and this is the main subject of this paper.
We show using abelian covers (Theorem \ref{covers}) 
that in general the ortho spectrum does not 
determine the systolic length 
but that (Theorem \ref{thm.McKean}) 
there are only finitely many possibilities 
for the the systolic length  for a given ortho spectrum.
As a corollary  we show that,  up to isometry,
there are only finitely many 
hyperbolic structures on a surface $\MMm$
that share a  given ortho spectrum.
To do this, following Wolpert \cite{Wolpert},
we use Mumford's  pre compactness criterion
for subsets of moduli space $\moduli$:
a subset $U \subset \moduli$ is pre compact,
if and only if,  the infimum of the  systole over $U$ is strictly positive.
Note here that when the surface $\Sigma$ has boundary, 
hyperbolic structures on $\Sigma$ we consider are assumed that all the components of the boundary are totally geodesic.
Moreover we define the moduli space $\moduli$ of $\Sigma$ to be the space of hyperbolic structures on $\Sigma$
 with the fixed perimeter.
 %sum of the lengths of the totally geodesic boundary fixed.
Thus we must show that 
given a hyperbolic structure $X$
then there is a uniform 
lower bound on the systolic length which only 
depends on the ortho spectrum of $X$
and its topological complexity.
Here, 
by \textit{topological complexity}
we mean the number of pairs of pants in a pants decomposition
(equivalently the area of $X$).
At the outset we hoped to obtain this result 
by using properties of the Poincar\'e series.
Unfortunately, our argument 
(see Paragraph \ref{small critical} for a detailed discussion)  
only worked when  the entropy of the geodesic flow was strictly less than $\frac12$.
This is problematic since many surfaces do not satisfy 
this hypothesis 
and indeed
when a surface has cusps the entropy is at least $\frac12$.
We did however obtain some partial  results 
which we include here 
as they may be of independent interest
(see  Paragraphs  \ref{planar surfaces} and \ref{small critical}).
Notably we show using the identities and a geometric limit argument
 that for a finite area $n$-punctured hyperbolic surface with totally geodesic boundary which is topologically a punctured disc, there are only finitely many isospectral pairs.
Finally, in  Paragraph \ref{mckean} 
we give a proof 
using a more combinatorial approach 
and some hyperbolic trigonometry
which can in principle be made effective
i.e. used to give an upper bound 
on the number of isospectral pairs for a given surface.

\subsection{Questions, remarks,  further work}
%\marginpar{Can you add the bounds they get?}
For a compact surface without boundary
Buser \cite[Chapter 13]{Buser} has in fact given upper bounds  in terms of the genus $g$
for the number of hyperbolic structures on a surface $\MMm$
that share a  given length spectrum.
These bounds were exponential of the square of $g$, namely $e^{720g^{2}}$.
They have recently been much improved by Parlier \cite{Parlier} to $g^{154g}$.
Since our proof depends on a compactness argument
we do not obtain any such bounds. 
It seems an interesting question to render 
our theorem \textit{uniform}
that is to find a bound on the number 
of hyperbolic structures on a surface $\MMm$
that share a  given ortho spectrum
purely in terms of the topological complexity of $\MMm$
as defined above.

In Section \ref{sec.covers} we give examples of hyperbolic structures
on a surface $\MMm$
with the same ortho spectrum but different systolic length.
Since the systolic length is
the infimum of the length spectrum 
they have different length spectra.
At the time of writing,
all known pairs of hyperbolic structures
with the same length spectrum are
\textit{commensurable}
(it is conjectured that this is always the case)
so that, by covering theory (see Section \ref{sec.covers})
they necessarily have the same ortho spectrum.
This suggests that the ortho spectrum 
is potentially a strictly weaker invariant 
than the length spectrum:
i.e. the length spectrum determines the ortho spectrum
but not vice versa.

Finally, Bridgeman \cite{Bridgeman} has  investigated the 
ortho spectrum for an ideal polygon in the hyperbolic plane,
showing how to obtain  certain celebrated  identities 
for the Rogers' dilogarithm from it.
Together with Dumas \cite{Bridgeman-Dumas} he went on to 
study a  related problem 
namely that of the  chord length distribution for the ideal triangle.
In a forthcoming paper \cite{MMc2019} we consider the 
problem of whether the ortho spectrum 
determines an ideal polygon up to congruence.

\subsection*{Acknowledgement}
Main part of this work was done during the visit of the second named author to Tokyo in early summer of 2019.
The visit was supported by JSPS KAKENHI Grant Number 16H02145 and 19K14525.
We would like to thank Ara Basmajian, Martin Bridgeman,  Yi Huang, and Hugo Parlier for helpful conversations. 
We thank in particular 
Dragomir Saric  whose comments
helped improve exposition of the geometric limit arguments 
in  \ref{planar surfaces}.
Finally, we would also like to thank the anonymous referees for helpful comments that improved the paper greatly.
\section{Ortho geodesics}

It is convenient to define ortho geodesics 
by using the action of  a Fuchsian group on the universal cover of $\MMm$.
A \textit{Fuchsian group} $\Gamma$ is a discrete subgroup of $\isomH$.
If $\Gamma$ is torsion free then 
the  quotient of $\HH$ 
by the action of $\Gamma$ 
is a surface $\MMm = \HH/\Gamma$
and $\pi_1(\MMm) \simeq \Gamma$.
The limit set $\Lambda(\Gamma)$  of $\Gamma$ 
is the smallest closed $\Gamma$-invariant subset 
of the ideal boundary $\partial \HH$
and, provided $\Gamma$
is not virtually abelian,
this is a perfect set.
The complement $\Omega(\Gamma)$ of the limit set 
is called the \textit{regular set} if
it is a (possibly empty) $\Gamma$-invariant open set.
Further, if $\Gamma$ is finitely generated and $\MMm$ 
does not have finite area then 
$\Omega(\Gamma)$  is dense
and consists of countably many open intervals.
If $\Gamma$ contains no parabolic elements 
then the orbits of the action of $\Gamma$ on $\Omega(\Gamma)$
are in 1-1 correspondence with the ends  of $\MMm$.
Thus, we have a 
$\Gamma$-invariant decomposition of the ideal boundary
of $\HH$ as
$$\dHH = \Lambda(\Gamma)  \sqcup \Omega(\Gamma).$$
We shall denote by $\partial \Omega$
 the set of all the points $a,b$ such that  the 
intersection of the interval  $[a,b] \subset \dHH$
with the limit set   $\Lambda$ is $\{a,b\}$.

\subsection{Convex core and ortho geodesics}

Given  $\Gamma$ finitely generated and $\MMm$ 
of infinite area there is a canonical way to associate
a subsurface $C(\MMm) \subset \MMm$ 
of finite area with totally geodesic  boundary
called the \textit{convex core}.
Let $C(\Lambda)\subset \HH$ be the convex hull of the limit set,
this is a  closed, $\Gamma$-invariant subset 
whose frontier consists of countably many complete geodesics.
The quotient $C(\MMm) := C(\Lambda)/\Gamma$ 
embeds naturally into $\MMm = \HH/\Gamma$.
By construction,
$C(\Lambda)$ is the universal cover of $\MMm$ and
the embedding induces an isomorphism
between $\pi_1(\MMm) \simeq \Gamma$ and
$\pi_1(C(\MMm))$.
In particular: 
\begin{prop} \label{bdy}
The components of the regular set,
i.e. the maximal intervals  in the complement of $\Lambda$, 
are in 1-1 correspondence with lifts of the boundary
geodesics of $\MMm$.
\end{prop}

%\begin{figure}
%\includegraphics[scale=1]{shearWithTree.pdf} 
%\caption{Convex hull of the limit set for a 2 generator 
%Fuchsian group}
%\end{figure}
Note that if $\gamma \subset \HH$ is a geodesic
with an endpoint in $\partial \Omega$ then
the corresponding geodesic in the surface $\MMm$
contains some boundary geodesic in its closure.
In what follows we will think of $C(\Lambda)$ 
as a generalized polygon and refer to the geodesics 
of $\partial C(\Lambda)$ as \textit{sides}.
We associate to pairs of distinct sides
an \textit{ortho geodesic} $\hat{\alpha^*} $,
which is just the unique common perpendicular joining
 the sides. 
 The image of  $\hat{\alpha^*} $ is a  geodesic arc  $\alpha^*$
which meets $\partial C(\MMm)$ perpendicularly
and this is an \textit{ortho geodesic on the surface $\MMm$}.
By definition, the lengths of $\hat{\alpha^*} $ and  $\alpha^*$
are the same and clearly the length of $\hat{\alpha^*} $
can be computed as a cross ratio of the endpoints of the
associated sides of $C(\Lambda)$.
We denote by $\mathcal{O}(\Sigma)$ the {\em ortho spectrum} of $\Sigma$, namely, the set of lengths of ortho geodesics counted with multiplicity.

\subsection{Identities for moduli}

One can compute two important numerical invariants  (moduli) 
of the hyperbolic structure from the ortho spectrum,
namely the perimeter (total length of the boundary) 
and the area.

\begin{thm}[Basmajian \cite{Bas}, Calegari \cite{Cal}]\label{thm.BK-C-boundary}
For each $n\in\mathbb{N}$, 
there exists a function $B_{n}:\mathbb{R}_{>0}\rightarrow\mathbb{R}_{>0}$ with the following property.
Let $M$ be a compact hyperbolic $n$-manifold with totally geodesic boundary.
Then $$\mathrm{Vol}(\partial M) = \sum_{\ell\in\mathcal{O}(M)}B_{n}(\ell).$$
If $\MMm$ is a surface with a single totally geodesic boundary component $\delta$. 
Then the above formula can be written explicitly by 
$$ \sum_{\alpha^*}  2 \sinh^{-1} \left( {1 \over \sinh(\ell(\alpha^*))} \right)  = \ell(\delta).$$
\end{thm}

\begin{thm}
[Bridgeman-Kahn \cite{BK}, Calegari \cite{Cal}]\label{thm.BK-C}
For each $n\in\mathbb{N}$, 
there exists a function $F_{n}:\mathbb{R}_{>0}\rightarrow\mathbb{R}_{>0}$ with the following property.
Let $M$ be a compact hyperbolic $n$-manifold with totally geodesic boundary.
%Let $\mathcal{O}(M)$ denote the ortho spectrum of $M$.
Then $$\mathrm{Vol}(M) = \sum_{\ell\in\mathcal{O}(M)}F_{n}(\ell).$$
\end{thm}
\begin{rmk}\label{rmk.cusped-version}
The proofs in \cite{BK} and \cite{Cal} only use the fact that generic geodesics eventually meet the boundary.
Hence their proof work not only for compact case but also the case where $M$ has cusps.
We use the cusped version later in Section \ref{planar surfaces}.
\end{rmk}
\begin{rmk}
Bridgeman-Kahn and Calegari gave different proofs of Theorem \ref{thm.BK-C}.
Given that the construction of Bridgeman-Kahn and of Calegari  are  different, 
so could be the functions $F_{n}$, 
but in \cite{MMc} we show that these functions are the same.
\end{rmk}

\section{Geometry of pairs of pants}

This section contains material necessary in the proof of  Theorem \ref{thm.McKean} below in particular certain relations between 
lengths of ortho geodesics in a pair of pants and the boundary lengths.

We begin by recalling a useful formula from hyperbolic trigonometry  (see for example \cite[Equation 7.18.2]{Beardon}) .
Choose  a pair of adjacent sides
in a  right angled pentagon
 and let  $a,b$ denote their respective lengths (see Figure \ref{fig.RApenta})
 then 
\begin{equation}
\label{eq.pentagon}\cosh d = \sinh a\sinh b.
\end{equation}
where $d$ is the length of the unique side not adjacent 
to either of our chosen pair of sides.
\begin{figure}[h]
\begin{center}
\includegraphics[scale=.3]{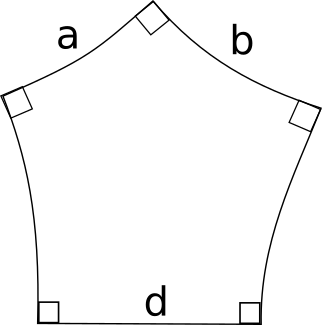} 
\end{center}
\caption{}
\label{fig.RApenta}
\end{figure}
\begin{defn}
A {\em foot} of an ortho geodesic is  a point of  intersection 
of the ortho geodesic with the
(totally geodesic)  boundary of the surface. 
Since the surface  is convex an ortho geodesic has 
exactly two feet.
\end{defn}

\begin{lem} \label{cord bounds}
Let $P$ be a pair of pants with boundary geodesics $\alpha, \beta, \gamma$.
Let $\tau$  be the unique  simple ortho geodesic
with both their feet on  $\gamma$. 
Then after possibly exchanging the labels on $\alpha,\beta$, we have
\begin{eqnarray}
\sinh(\ell(\tau)/2) &\leq& \frac{\cosh(\ell(\alpha)/2) }{\sinh(\ell(\gamma)/4)}
 \end{eqnarray}
\begin{proof}
The feet of $\tau$ divide $\gamma$ into two intervals one of which 
is of length $y \geq \ell(\gamma)/2$. 
After possibly exchanging the labels on $\alpha,\beta$ 
one may assume that  there is an embedded right angled pentagon
in $P$ with a pair of adjacent sides of length
$y/2,\, \ell(\tau)/2$ and the non adjacent side has length $\ell(\alpha)/2$.
Applying formula  (\ref{eq.pentagon}) one obtains
$$ \cosh(\ell(\alpha)/2)  
= \sinh(y/2)\sinh(\ell(\tau)/2) \geq \sinh(\ell(\gamma)/4)\sinh(\ell(\tau)/2).$$
So 
$$ \sinh(\ell(\tau)/2) \leq \frac{\cosh(\ell(\alpha)/2) }{\sinh(\ell(\gamma)/4)}
$$

% It follows that
% $$ \exp(\ell(\alpha)/2)  
% \geq  (\ell(\gamma)/2) \times \exp(\ell(\tau)/2)$$
\end{proof}

\end{lem}

\begin{lem}
\label{ortho determines alpha}
Let $P$ be a pair of pants with boundary geodesics $\alpha, \beta, \gamma$.
Let $\tau$ and  $\tau'$ be ortho geodesics with both their feet on  $\gamma$
%We further let $\tau_{n}$ denote the ortho geodesic based on $\gamma$ going around $\alpha$ $n$-times (see Fig. ***).
such that $\tau$ is simple  and $\tau'$ 
goes round $\alpha$ exactly once (see Figure \ref{two pants with curve}).
%We denote by $a, t$ and $t_{n}$ the length of $\alpha, \tau$ and $\tau_{n}$ respectively.
Then we have
\begin{eqnarray}
 \cosh(t'/2) & = & 2\cosh(a/2)\cosh(t/2),\\
 t' & \leq & \ell(\gamma) + 2t\label{eq.t'},
\end{eqnarray}
where $a, t$ and $t'$ denote the lengths of $\alpha, \tau$ and $\tau'$ respectively.
\end{lem}

\begin{figure}[h] 
\begin{center}
\includegraphics[scale=.4]{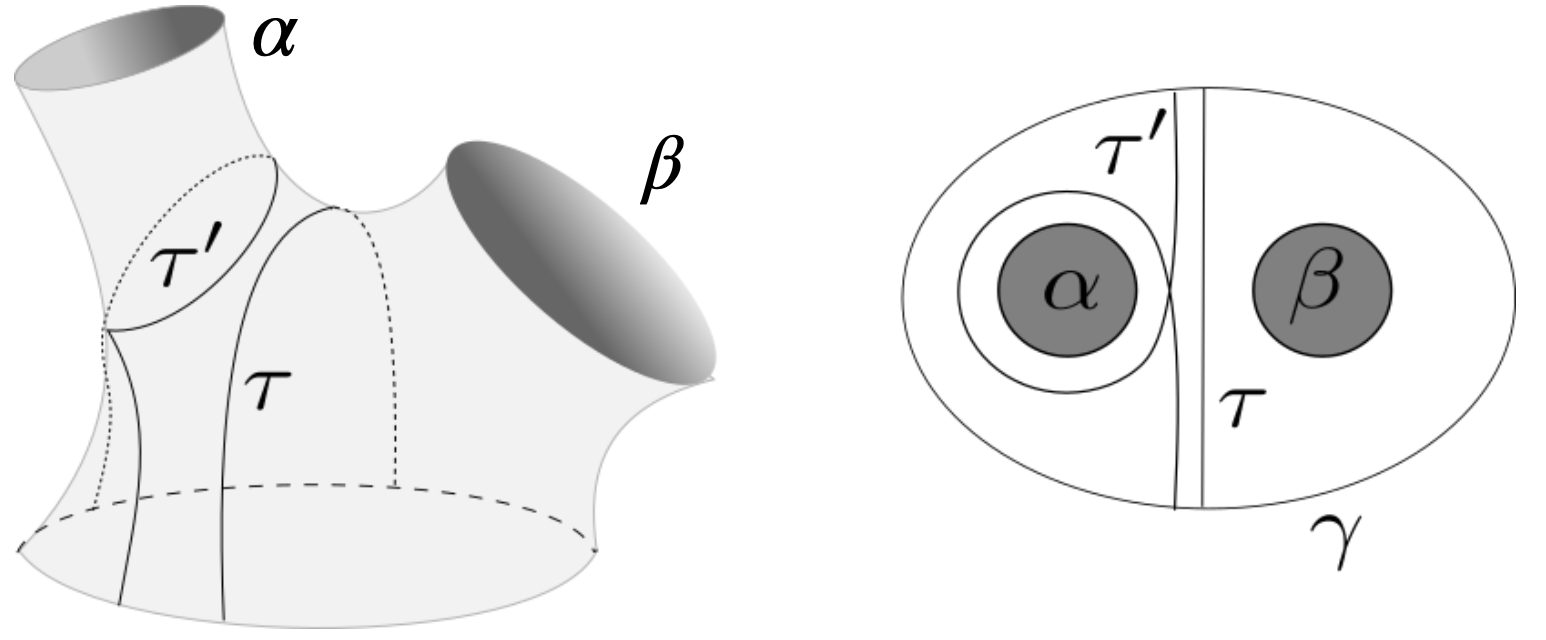} 
\end{center}
\caption{}
\label{two pants with curve}
\end{figure}
\begin{proof}
We cut $P$ along a set of curves 
to obtain a collection of  four  right angled pentagons.
So cut $P$ along $\tau$, 
and  the shortest ortho geodesics between $\alpha$ and $\gamma$, between $\beta$ and $\gamma$ ,
and $\alpha$  and $\beta$.
One of these pentagons has a pair of adjacent edges of lengths 
$a/2$ and $b$, the length of the side not adjacent to this pair
is $t/2$ (see Figure \ref{fig.estimates}).
 \begin{figure}[h]
 \label{fig.estimates}
 \begin{center}
 \includegraphics[scale=.4]{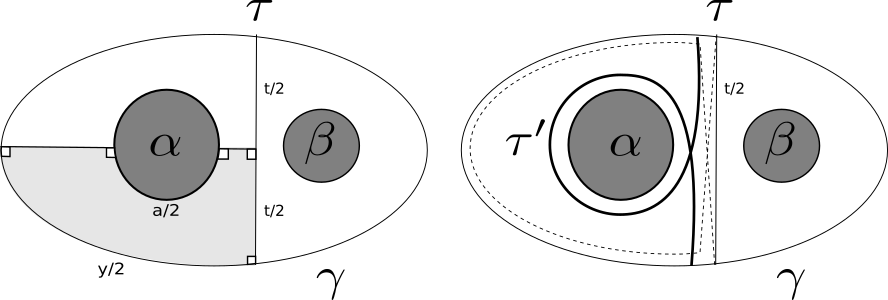} 
 \caption{}
 \label{fig.estimates}
 \end{center}
 \end{figure}
 By Eq. (\ref{eq.pentagon}), we have
\begin{equation}
\label{eq.1}\cosh(t/2) = \sinh(a/2)\sinh(b).
\end{equation}
%\begin{figure}[ht]
%\label{pants with curve}
%\begin{center}
%\includegraphics[scale=.4]{pentagon1.png} 
%\end{center}
%\end{figure}
There is a double cover of the pair of pants such that $\tau'$ lifts 
to a simple curve $\tilde{\tau'}$ 
and in this surface there is an embedded right angled pentagon
with a pair of adjacent edges of lengths $a$ and $b$
(see Figure \ref{fig.double-cover}).
The length of the non adjacent edge is $t'/2$.

\begin{equation}
\label{eq.2}\cosh(t'/2) = \sinh(a)\sinh(b).
\end{equation}
By Eq. (\ref{eq.1}) and Eq. (\ref{eq.2}), we have
\begin{align*}
\cosh(t'/2) &= \sinh(a)\frac{\cosh(t/2)}{\sinh(a/2)}\\
                     &= 2\cosh(a/2)\cosh(t/2).
\end{align*}
Eq. (\ref{eq.t'}) follows immediately from Figure \ref{fig.estimates}.
\begin{figure}[h]
\begin{center}
\includegraphics[scale=.4]{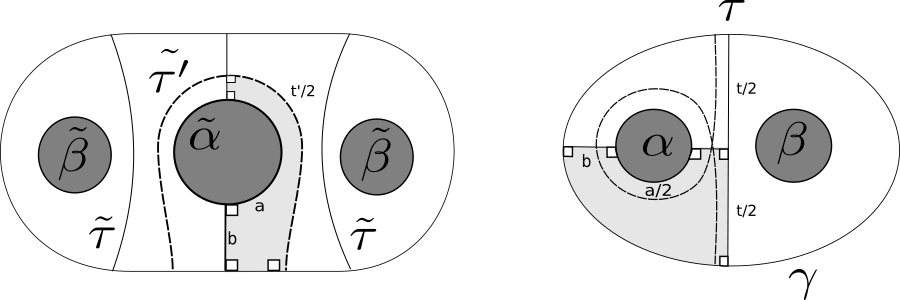} 
\end{center}
\caption{}
\label{fig.double-cover}
\end{figure}
\end{proof}
\section{Strong rigidity for the one holed torus}

We now prove that the  ortho spectrum of a one holed torus determines the hyperbolic structure up to isometry.
The analogous result for the length spectrum is due to 
Buser and Semmler \cite{BS}.
For more general surfaces such a strong rigidity result is not possible 
as one can construct surfaces using abelian covers which
are homeomorphic (see Section \ref{sec.covers}), have the same ortho spectrum  
but are not isometric.
Thus, in the general case we will  prove McKean-type theorem 
for the  ortho spectrum: 
that is we show that given a hyperbolic structure on the surface then there 
are only finitely many hyperbolic metrics with the same  ortho spectrum.

\begin{lem}\label{lem. shortest meets once}
Let $\MM$ be a hyperbolic one-holed torus with totally geodesic boundary.
Let $\alpha$ be the unique  simple closed geodesic disjoint
 from one of the shortest ortho geodesics $\tau$.
Then the shortest ortho geodesic that crosses $\alpha$
must  meet $\alpha$ in a single point.
\end{lem}
\begin{proof}
Let $b$ be an ortho geodesic that crosses $\alpha$ more than once.
Note that $\ell(\tau)\leq\ell(b)$.
 We will show that there is an ortho geodesic $x$ 
 which crosses $\alpha$ exactly once and is shorter than $b$.
 
 Let $\tMM \rightarrow \MM$
 denote the infinite cyclic cover in Figure \ref{fig.holed-torus}.
 \begin{figure}[h]
 \begin{center}
 \includegraphics[scale=.4]{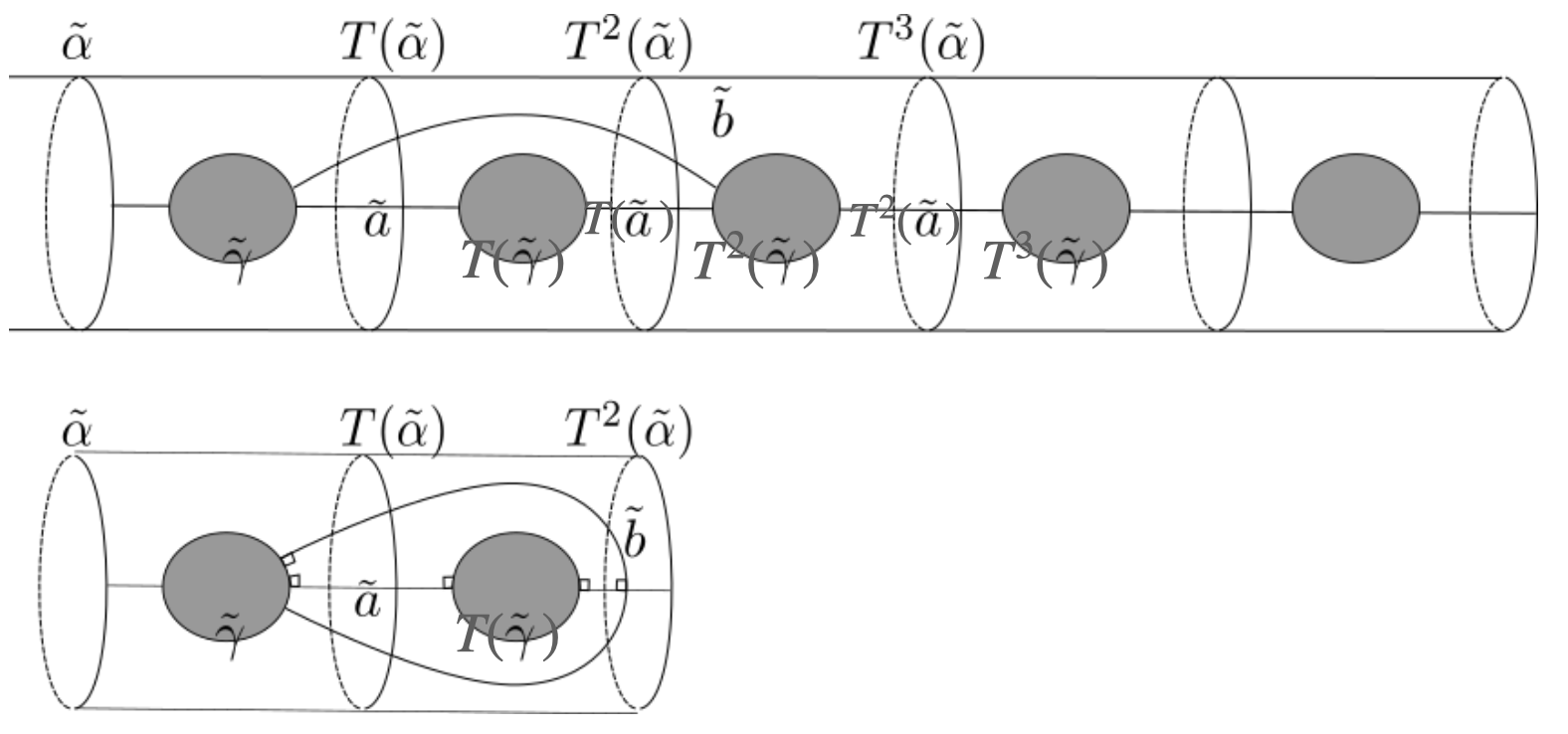} 
 \end{center}
 \caption{}
 \label{fig.holed-torus}
  \end{figure}
 The boundary geodesic lifts to a closed geodesic $\tilde{\gamma}$, and
 $\alpha$ lifts to a closed geodesic $\tilde{\alpha}$.
 If $T : \tMM  \rightarrow \tMM $
 is the  generator of the group of deck transformations
 then $\tilde{\alpha}$ and $T(\tilde{\alpha})$
 bound a fundamental domain $D$ for the action of the deck transformations.
 There is a lift $\tilde{b}$ of $b$ to $\tMM $
 which, without loss of generality exits the fundamental domain $D$
 by crossing $T(\tilde{\alpha})$.
 
 By hypothesis $b$ meets $\alpha$ more than once,
 so $\tilde b$ must meet a translate of $\tilde{\alpha}$ a second time
and this translate must be  either $T(\tilde{\alpha})$ or $T^2(\tilde{\alpha})$.
In either case there is an arc
 $c$ which  minimises the distance between
 $\tilde{b}$ and $T(\tilde{\gamma})$.

%Let $a$ be an arc that realises the minimum distance between $\tilde{b}$ and $T(\tilde{\gamma})$.
We construct  a  right angled pentagon
with a pair of  adjacent edges of length less than 
$\ell(c), \ell(b)/2$ 
and the edge disjoint  from  this pair  is  an ortho geodesic $x$ %of length  $x$,
whose projection to $\MM$ meets $\alpha$ exactly once.
It is always possible to construct the pentagon such that 
$\ell(c) \leq  \ell(\tau)/2 $.

Now, since  $\cosh$ and $\sinh$ are monotone increasing functions on $\mathbb{R}_{\geq 0}$,
 by Eq. (\ref{eq.pentagon}),
\begin{align*}
\cosh(\ell(x)) &\leq \sinh(\ell(\tau)/2) \sinh(\ell(b)/2)\\
&\leq \sinh(\ell(b)/2)^{2}\\
&< \sinh(\ell(b)/2)^{2} + \cosh(\ell(b)/2)^{2}\\
&= \cosh(\ell(b)).
\end{align*}
Hence we have $\ell(x) < \ell(b)$.

%In the second case we have a right angled pentagon of edge length anti-clock wisely
%$\{\ell(x), *, \leq\ell(\tau), \leq\ell(\beta),*\}$.
%Again by the same calculation using Eq. (\ref{eq.pentagon}) as above, 
%we have $\ell(x)\leq\ell(\beta)$.
\end{proof}

\begin{thm}
Let $T_{1},T_{2}$ be a pair of  %convex
 hyperbolic structures 
of area $2\pi$
on the one-holed torus  such that the boundary is 
a closed geodesic.
Then $T_{1}$ and $T_{2}$ are isometric if and only if $\mathcal{O}({T_{1}}) = \mathcal{O}({T_{2}})$.
\end{thm}

We begin by using  Basmajian's identity to show that boundary lengths are equal
then, as in Buser-Semmler,  we show that the structures $T_i$ determine  (essentially)  the same Fenchel-Nielsen parameters for a pair of simple geodesics $\alpha_i$. 
More precisely, 
$\ell(\alpha_1) = \ell(\alpha_2)$
and the twist parameters are the same up to change of sign.

\begin{proof}
Let $\gamma$ denote a simple loop freely homotopic to the boundary 
of the one-holed torus.
The free homotopy class of $\gamma$ is invariant under the action of the 
orientation preserving homeomorphisms of the one-holed torus.
By Basmajian \cite{Bas} $\ell(\gamma)$, 
 the length of the unique geodesic in this homotopy class,
can be determined from just the ortho spectrum 
 and so it is independent  of the choice of structure 
  $T_{i},\, i = 1,2$.
  
%Let $a$ denote the length of the shortest ortho geodesic.
Let $\tau$  be a simple ortho geodesic and  denote its length by  $t$.
For $i=1,2$ there is a unique simple closed geodesic 
$\alpha_i$ on  $T_i$
disjoint from  $\tau$ and we have 
\begin{eqnarray} \label{pants formula}
\cosh(\ell(\alpha_i)/2 ) = \sinh(t/2)\sinh(\ell(\gamma)/2),
 \end{eqnarray}
To prove (\ref{pants formula}) we cut along $\alpha_i$ to obtain a pair of pants
$P_{i}$ with two boundary geodesics 
of length $\ell(\alpha_i) $ and another of length $\ell(\gamma)$.
There is an embedded right angled pentagon 
with an adjacent pair of sides of lengths $t/2, \ell(\gamma)/2$
and the non adjacent edge has length $\ell(\alpha_i)/2$
and Eq.(\ref{pants formula}) follows from Eq.(\ref{eq.pentagon}).

It follows that if $\ell(\tau_1) = \ell(\tau_2)$  ortho geodesic on $T_i$
then the closed simple geodesics $\alpha_i$ have the same length.
It remains to show that the Fenchel-Nielsen twist parameters are the same
for $\alpha_1$ and $\alpha_2$.

Let $P_{i}$ be the pair of pants we get by cutting along $\alpha_{i}$, and 
$\mathcal{O}(P_{i},\gamma)$ the ortho spectrum of $P_i$ of ortho geodesics whose feet are on $\gamma$.
Since the geometry of pairs of pants are determined by the boundary lengths, we have $\mathcal{O}(P_{1},\gamma) = \mathcal{O}(P_{2},\gamma)$.
Consider the  ``set-wise" difference 
$\mathcal{O}(T_{i})\setminus\mathcal{O}(P_{i},\gamma)$.
What we obtain by doing this is the lengths of 
those ortho geodesics which intersects $\alpha_{i}$ for $i = 1,2$.
In particular, 
we can determine $\ell(\tau^{1}_i)$ the length of the shortest such geodesic
$\tau^{1}_i$ for $i = 1,2$.
By Lemma \ref{lem. shortest meets once},
the ortho geodesic $\tau^{1}_i$ meets $\alpha_{i}$ just once.
Now any pair of curves that each meets $\alpha_i$ just once 
are related by a Dehn twist,
and since the length of an ortho geodesic is strictly convex 
along a Fenchel-Nielsen twist \cite{Ker}
there are at most two possibilities
for the choice of $\tau^{1}_i$.
We normalise the Fenchel-Nielsen twist
so that the parameter is $0$ 
when the length of $\tau^1_i$ is minimal.
Using Kerckhoff's formula \cite{Ker} for the variation
of length along a Fenchel-Nielsen twist this occurs exactly when 
$\alpha_i$ and $\tau^1_i$ meet at right angles.
It is not difficult to see that when there are two choices 
for $\tau^1_i$ this  yields a pair of  surfaces 
which are isometric 
and the corresponding twist parameters differ 
only in their sign.

%Let $\Gamma_i < \mathrm{isom}(\HH)$ 
%be a Fuchsian group uniformizing $T_i$
%we identify $\Gamma_i$ with the fundamental group
%of the holed torus with base point $\alpha_i \cap \tau^{1}_i$.
%Since $\Gamma_i$ is torsion free it lifts 
%to group $\hat{\Gamma}_i < \mathrm{SL}(2,\RR)$.
%The construction in the preceding passage
%yields a preferred set of generators for 
%$\hat{\Gamma}_i$ namely $\alpha_i$ and $\tau^{1}_i$
%and we have
%$$ \mathrm{tr}\,   \alpha_i  = 2 \cosh(\ell(\alpha_i),\, 
%\mathrm{tr}\,   \tau^{1}_i  = 2 \cosh(\ell(\tau^{1}_i).$$

\end{proof}

\section{Rigidity in general: McKean's Theorem}
Sunada \cite{Sunada} gave an ingenious geometric construction 
allowing one to construct pairs of hyperbolic surfaces $Y_1,Y_2 \in \moduli$
with the same length spectrum but which are not isometric.
On the other hand, by a theorem of McKean \cite{McKean},
 for any given surface $Y\in \moduli$ there can only be finitely
many surfaces with the same spectrum as $Y$.
In this section we prove a version of McKean's theorem for the ortho spectrum:
\begin{thm}\label{thm.McKean}
Let $\Sigma$ be a compact surface with single boundary component $\alpha_0$ and
$X$  a hyperbolic structure on $\Sigma$ with totally geodesic boundary.
Then there are finitely many choices for a hyperbolic structure $Y$ on $\Sigma$
such that $\mathcal{O}(X) = \mathcal{O}(Y)$.
\end{thm}

Whilst, as we shall see in  Section \ref{sec.covers}, 
it is much easier to find pairs of hyperbolic surfaces $Y_1,Y_2 \in \moduli$
with the same ortho spectrum but which are not isometric,
it turns out that the analogue of McKean's Theorem 
for the ortho spectrum is quite subtle.
So, before giving a proof of Theorem \ref{thm.McKean} in full generality, 
we discuss  the general strategy of the proof and how  the Bridgeman-Kahn identity and Poincar\'e series can be used to obtain the result in two special cases.

\subsection{Wolpert's strategy} \label{wolperts strategy}
We will employ a strategy due to Wolpert \cite{Wolpert}.
His argument is for the length spectrums, however we may almost identically obtain the finiteness as follows.
%He observed that, 
Since the ortho spectrum is discrete,
for any compact set  $\hat{B}$ 
of $\teich$  the Teichm\"uller space of the surface,
the set of hyperbolic structures  in $\hat{B}$ 
with the same  ortho spectrum as $Y$ is finite.
Now ortho spectrum depends only on the 
point determined by the hyperbolic surface in $\moduli$.
Thanks to Theorem \ref{thm.BK-C-boundary}, provided we have given the ortho spectrum, 
we only need to consider hyperbolic structures on $\Sigma$
 with the fixed perimeter.
So if $B$ is the projection of $\hat{B}$  to $\moduli$
we may apply Mumford's Criterion for 
pre compactness namely:
a subset $B \subset \moduli$ is pre compact,
if and only if,  the infimum of the  systole over $B$ is strictly positive.
So if we can  bound the  systole length away from zero, 
then Mumford's Compactness Theorem  implies Theorem \ref{thm.McKean}.

\subsection{Separating geodesics}\label{planar surfaces}

Recall that a simple closed curve on a surface is called {\em essential} if it is not homotopic to a point, a puncture nor a boundary component.
An essential simple closed curve is called separating if it separates the (connected) surface into two components.

We argue by contradiction,  
suppose that there is a sequence of hyperbolic structures $Y_{i}$ on $\Sigma$ with all boundary components totally geodesic so that $\mathcal{O}(Y_{i}) = \mathcal{O}(X)$, 
and that the systole length $\mathrm{sys}(Y_{i})$ converges to zero.
After passing to  a subsequence if necessary,
we may suppose there is a simple closed curve $\alpha\subset\Sigma$ 
with $\ell_{Y_{i}}(\alpha)\rightarrow 0$.
Suppose that this curve is separating 
and that one of the subsurfaces has no geodesic boundary 
component.
By Basmajian's identity the perimeter $P$  of $Y_i$ does not depend on $i$
so there is some component of the boundary,
the \textit{good boundary component},
for which the length is bounded below away from $0$.
We choose  $x_i$  a   base point  for $Y_{i}$  
on the good boundary component
and consider the limit of $(Y_{i}, x_i)$.
After passing to a subsequence if necessary 
we may suppose that this sequence converges to a surface
$(Y_{\infty}, x_\infty)$ 
with at least one cusp and which has, 
since $\alpha$ is separating,
area at most that  $Y_0$ 
minus the area of a pair of pants i.e.
$2\pi$.
Intuitively, 
there is a subsurface  which is getting further and further 
from the base point. 
More concretely, 
by the Collar Lemma there exists $m$ such that for all $i>m$ 
there is a collar round $\alpha$ on $Y_i$ of width $L$
and as a consequence the length of any ortho geodesic that meets
$\alpha$  satisfies
$$\ell_{Y_i}(\tau) > 2L.$$

%\begin{figure}[h]
%\centering
%\begin{subfigure}{.5\textwidth}
%\centering
%\includegraphics[scale=.1]{pinching.png} 
%\end{subfigure}
%\begin{subfigure}{.5\textwidth}
%\centering
%\includegraphics[scale=.1]{pinched.png} 
%\end{subfigure}
%\caption{}
%\label{fig.tau_n}
%\end{figure}

 \begin{figure}[h]
  \centering
  \subfloat  
{\includegraphics[scale=.25]{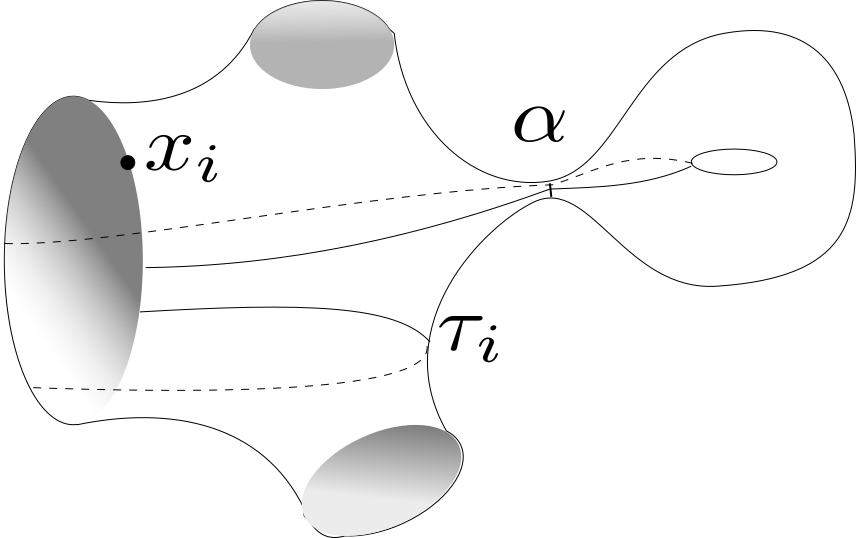} 
}
 \hfill
  \subfloat
  {\includegraphics[scale=.25]{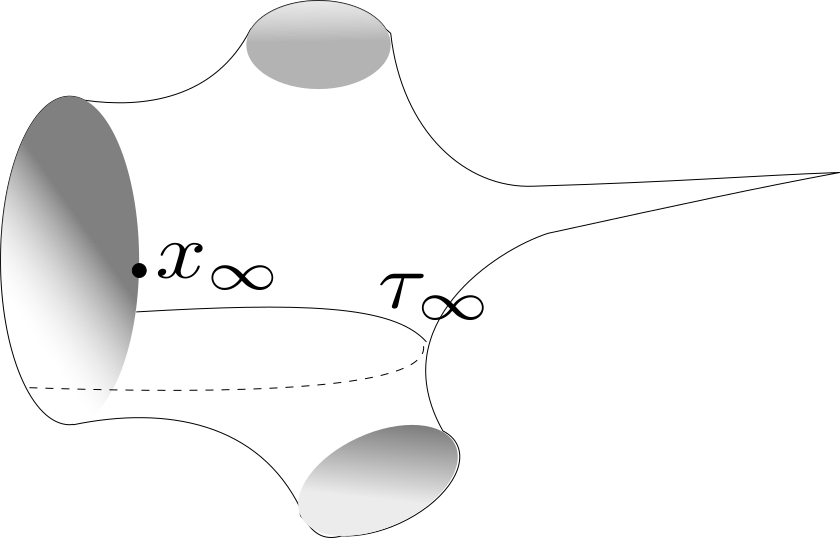} 
}
  \caption{Taking a geometric limit.
  The curve $\alpha$ has been pinched to a cusp
  in the   limit surface $Y_\infty$ on the right. 
  The length of an ortho geodesic that crosses $\alpha$ 
  goes to $\infty$. 
  The sequence of ortho geodesics $\tau_i$ limits to $\tau_\infty$.
    }

\end{figure}

Now by Theorem \ref{thm.BK-C} (and Remark \ref{rmk.cusped-version}) for each $Y_i$
$$\mathrm{Vol}(Y_i)  = 
\mathrm{Vol}(Y_0) 
= \sum_{\ell\in\mathcal{O}(Y_0)}F_{2}(\ell).$$
and so
\begin{enumerate}
\item 
since each term $F_{2}(\ell)$ is positive 
and the series is convergent,
for any $L>0$  one has 
$$\sharp \{\ell < L,\,  \ell\in\mathcal{O}(Y_i ) \} =  
\sharp \{\ell < L,\,  \ell\in\mathcal{O}(Y_0 )  \} < \infty $$
\item
 there is $L>0$ such that 
$$ \sum_{\ell<L}F_{2}(\ell) 
> \mathrm{Vol}(Y_0) - 2\pi.$$
\end{enumerate}
Let $\gamma\subset Y_i$ denote the boundary component of $Y_i$ 
containing $x_i$.
Evidently since the  perimeter $P$  of $Y_i$ does not depend on $i$
we have a uniform bound on the length of $\gamma$ namely
$$\ell_{Y_i}(\gamma) \leq P.$$
%Thus $(Y_i,x_i)$ is  covered  by one half of a hyperbolic  annulus
%with a core geodesic $\tilde{\gamma}$ 
%which has  length less than $P$.
The subset of the component $Y_i \setminus \{\alpha\}$ determined by $x_i$
containing the ortho geodesics  $\tau$ of length less than  $L$ 
is contained in a subset of $Y_i$ 
for which the diameter is bounded by $L + P$.
% the image of a (half) collar 
%of width  $L$ around the core geodesic $\tilde{\gamma}$.
So for $\epsilon >0 $ we have uniform bounds,
independent of $i$,
for the number of balls of radius $\epsilon$ 
needed to cover the subset of  
this  component which  contains ortho geodesics  $\tau$  of length less than  $L$.
Given such a uniform bound
it is not difficult to show that 
for each $\ell \in \mathcal{O}(Y_0 ) \cap [0,L]$
there is a sequence of  ortho geodesics
$\tau_i(\ell) \subset Y_i$  of length $\ell$
which converge to an ortho geodesic 
$\tau_\infty(\ell) \subset Y_\infty$
and that reciprocally every such ortho geodesic on 
$Y_\infty$ is obtained in this way.
Since, as we have observed above,
there are only finitely many ortho geodesics 
of length less than $L$ on $Y_i$
we can pass to a subsequence of $Y_i$
for which this convergence is simultaneous
that is, for each $\ell \in \mathcal{O}(Y_0 ) \cap [0,L]$
$$\tau_i(\ell)  \rightarrow \tau_\infty(\ell) \subset Y_\infty.$$%[x + k*ty2 for x in row ][x + k*ty2 for x in row ][x + k*ty2 for x in row ]
Now we have a contradiction as
$$\mathrm{Vol}(Y_\infty ) 
\geq \sum_{\ell<L}F_{2}(\ell) 
 > \mathrm{Vol}(Y_0) - 2\pi  
 \geq \mathrm{Vol}(Y_\infty ).$$
 In particular, for punctured discs, where every essential simple closed curve is separating and one cut piece never contains boundary, 
 the argument in this Paragraph proves Theorem \ref{thm.McKean}.

\subsection{Surfaces with small critical exponent}\label{small critical}
We next observe that under the assumption that the Hausdorff dimension is less than $1/2$,
we can prove Theorem \ref{thm.McKean}  using properties of  Poincar\'e series.
From standard calculus it is easy to see that for $x$ large
\begin{eqnarray}\label{estimate}
\sinh^{-1} \left( {1 \over \sinh(x)} \right)  = 2e^{-x} + O(e^{-3x}),
\end{eqnarray}
and, in particular, the following Poincar\'e series converges
for $h=1$:
\begin{eqnarray}\label{poincare series}
\sum_{\alpha^*}  \exp( -h\ell(\alpha^*))
\end{eqnarray}
More generally,
by work of Patterson, Sullivan and Parker the series converges for all $h$ strictly greater 
 than the critical exponent (dimension of the limit set),
 see the Appendix for an exposition of this fact.

\begin{thm}
Let $\delta > 0$ and 
$Y_n \in \moduli$ be a sequence of 
compact hyperbolic surfaces with totally geodesic boundary
such that  the Hausdorff dimension of the limit set (critical exponent)
is less than $\frac{1}{2} - \delta$.
Then the set of surfaces $Y_n$ with the same ortho spectrum as $Y$ is finite.
\end{thm}

%Our proof uses a ``converse" to the Collar Lemma 
%due to Buser \cite{Buser}.
%Recall that the Collar Lemma says that in a 
%hyperbolic surface the lengths 
%of two closed geodesics $\alpha,\beta$
%satisfy the inequality
%\begin{eqnarray}
%\sinh \frac{\ell(\alpha)}{2} \sinh \frac{\ell(\beta)}{2} \geq 1,
%\end{eqnarray}
%in particular, if $\alpha$ is short
%$$ \ell(\beta) \geq -2\log\,\ell(\alpha).$$

%Buser's converse says that there is a constant $C_0$
%such that there is a geodesic of length
%$$ \ell(\beta) \leq  C_0 - 2\log\,\ell(\alpha).$$

\begin{proof}
We show that the systole is bounded over the 
set of surfaces $X$ with the same ortho spectrum as $Y$ 
using the function
$$\moduli \rightarrow \RR^+,\,
 X \mapsto \sum_{\alpha^* \subset X}  \exp( -h \ell(\alpha^*)).$$
We  control the length of the systole  using the following observation:

CLAIM: if the systolic length tends to $0$ as $n \rightarrow \infty$
then there is a sequence of closed geodesics $\alpha_n$
and a  sequence of distinct ortho geodesics $\tau_k$ 
and $M > 0$ such that:
\begin{enumerate}
\item $\ell(\alpha_n) \rightarrow 0$.
\item $\ell(\tau_k) \leq k \ell(\alpha_n) -2\log(\ell(\alpha_n)) + M$.
\end{enumerate}
Assuming the claim we choose $h < 1/2$ 
and strictly greater than the critical exponent.
We have
\begin{equation}
 \sum_{\alpha^* \subset X}  \exp( -h\ell(\alpha^*))
 \geq 
  \sum_{k \geq 0}  \exp(-h\ell(\tau_k) )
\geq
  \frac{C \ell(\alpha_n)^{2h}}{1 -  \exp(- \ell(\alpha_n) )}
  \sim 
  C  \ell(\alpha_n)^{2h-1}
\end{equation}
And thus  $\ell(\alpha_n) \not \rightarrow 0$   since $2h-1 < 0$
and the Poincar\'e series converges.

To prove the claim we proceed as follows.
We begin by remarking that there is at least one boundary component
$\gamma$ say for which $\ell(\gamma) \not \rightarrow 0$.
For, if all the boundary lengths went to zero then,
by the Collar Lemma, 
the length of the shortest ortho geodesic would tend to infinity
and by hypothesis this length is independent of $Y_n$.
In fact, by Basmajian's identity, 
we may suppose that 
$$\ell(\gamma) \geq \ell(\partial \MMm)/N$$
where $N$ is the number of boundary components.

%For each $Y_n$ we choose a pants decomposition $P_n$
%which contains all the short curves (length less than Margulis' constant)
%and such that all the curve lengths are less than the Bers' constant.
%Note that there are finitely many distinct topological types of pants decompositions.
%Hence after passing to a subsequence if necessary 
%we may assume that all the $P_n$ are of the same topological type.

Let $\alpha_n$ be a curve %in the pants decomposition
such that $\ell(\alpha_n) \rightarrow 0$.
%To start with we  suppose that, $\forall n$,
%the boundary component $\gamma$ 
%and $\alpha_n$ are boundaries of a pantalon in $P_n$.
Choose a point on $x \in \gamma$ and construct  a piecewise geodesic curve 
by traveling along the shortest geodesic arc  $a_n$ to $\alpha_n$,
going round $\alpha_n$ $k$ times,
and then returning to $x$ by going back along the shortest route again.

The lengths of these geodesics satisfy
\begin{equation}
\ell(\tau_k)
\leq k\times  \ell(\alpha_{n}) + 2 \ell(a_n)
\end{equation}
for all $k \geq 0$.

%\marginpar{Thanks, we should give a reference}
%\textit{finish by bounding $a\ldots$}

We now show that the length $\ell(a_n)$ is bounded from above by $M/2 -\log(\ell(\alpha_n))$ where $M$ depends only on the topology of $\MMm$.
This follows from the fact that for any given $\epsilon>0$ there is a universal constant $M_\epsilon$ such that for any $X\in\moduli$, the diameter of the $\epsilon$-thick part of $X$ is bounded from above by $M_\epsilon$.
One can see this by looking at the $\epsilon/4$-neighbourhood of a diameter realizing path. The neighbourhood must be embedded and the standard hyperbolic geometry shows that the area is $\text{diameter}\times 2\sinh(\epsilon/4)$.
The hyperbolic area of $X$ is $2\pi|\chi(S)|$ and hence the diameter must be bounded.
We may suppose that there exists $\epsilon>0$ such that for any $n$,
$a_n$ travels in the $\epsilon$-thick part of $Y_n$ until it reaches the collor of $\alpha_n$.
For otherwise, there must be a sequence of simple closed curves $\alpha'_n\subset Y_n$ with length $\ell(\alpha'_n)\rightarrow 0$ as $n\rightarrow\infty$ so that $\alpha'_n$ intersects with $a_n$.
In this case we just replace $\alpha_n$ with $\alpha'_n$.
Thus we get the constant $M:=2M_\epsilon$ so that the length of $a_n$ contained in the $\epsilon$-thick part of $Y_n$ is bounded above by $M/2$.
%We may suppose that the boundary lengths of pairs of pants that $a$ passes through are bounded from below,
%for otherwise we can choose 
%a curve  in the pants decomposition   whose length converges to $0$
%and which is closer to the boundary than $\alpha_n$.
%If the lengths of the boundary curves of a pair of pants are bounded from below and above, the distance between the boundary curves are also bounded (see Eq. (\ref{eq.pentagon}) above from Beardon \cite{Beardon}).
%The last pair of pants $P$ that $a$ passes through contains $\alpha$ as a boundary component.
%Since the number of pairs of pants that $a$ passes through is bounded by the topology of $\MMm$, 
The term $-\log(\ell(\alpha_n))$ corresponds to  the depth of the collar of $\alpha_n$.
\end{proof}

\subsection{McKean's Theorem in general} \label{mckean}

%\begin{thm}
%Let $X$ be a hyperbolic structure on a surface  $\Sigma$
%with a single geodesic boundary component $\alpha_0$.
%Then there are finitely many choices for a hyperbolic structure $Y$ on $\Sigma$
%such that $\mathcal{O}(X) = \mathcal{O}(Y)$.
%\end{thm}
We now give a proof of Theorem \ref{thm.McKean} %in Paragraph \ref{}
i.e. given a hyperbolic structure on the surface then there 
are only finitely many hyperbolic metrics with the same  ortho spectrum.
As before, following Wolpert \cite{Wolpert},
it suffices to show that 
the set of surfaces with the same ortho spectrum 
remains in a pre compact subset of moduli space $\moduli$.
So, again by Mumford's criterion, 
we will show that there is a strictly positive lower bound for the systole.
Our bound depends on a 
function which we will call 
the \textit{hyperbolic granulosity} 
of a discrete subset $S$ of the positive reals
$$ 
\text{granulosity}_S(L) := \inf\left\{ \frac{\cosh( y/2)}{\cosh( x/2)}~\middle|~  x,y\in S,x < y < L \right\}.$$
This is a monotone decreasing function of  $L \in \RR_+$,
bounded below by $1$
and is a measure of the smallest gap in $S$ in the interval $[0,L]$.
The hyperbolic granulosity appears naturally 
from the trigonometric identities 
we use to determine lengths of closed geodesics  
from those of  ortho geodesics.

Note that the ortho spectrum does not determine the systole's length. 
For example, in Section \ref{sec.covers} we construct pairs of surfaces with the same
ortho spectrum but different  systolic length.
Given such examples, it seems difficult to render such a result \textit{effective}
that is to give an upper bound for the number of
such hyperbolic structures which depends only on the topological type of the surface.

\begin{proof}[Proof of Theorem \ref{thm.McKean}]
Following Wolpert \cite{Wolpert}:
we suppose that there is an infinite family of pairwise non isometric
hyperbolic structures $Y_n$  with the same ortho spectrum as $X$.
We will show that there is a  strictly positive lower bound for the systole.

Let $B$ denote the Bers constant for $\Sigma$.
For each $Y_n$ choose a pants decomposition  $\mathcal{P}_n$
which satisfies the $\ell(\gamma) \leq B$ for all  
$\gamma \in \mathcal{P}_n $.
We make this choice since 
this guarantees that 
if $\ell(\alpha) \rightarrow 0$ as $n \rightarrow  \infty$
then $\alpha \in \mathcal{P}_n$.
%We will obtain a contradiction by
%bounding  $\ell(\gamma)$
%below independently of $n$.
We will bound the length $\ell(\gamma)$ from below independently of $n$.
This means that $Y_n$ remains in a compact
subset of moduli space and so must be eventually constant
since the ortho spectrum is discrete.

For each $n$ we construct a  \textit{rooted adjacency graph} $\Gamma_n$
as follows.
The  vertices of $\Gamma_n$ are the pants in $\mathcal{P}_n$,
a pair of vertices  is joined by an edge if they satisfy the obvious relation
and the root of $\Gamma_n$  
is the vertex corresponding to  the pants which meets the boundary.
Fix $n$ and choose a
rooted spanning tree $T$ for $\Gamma_n$.
The depth of the spanning tree is bounded 
above by the number of pants 
which does not depend on  $n$.
Let $P$ be a vertex of $T$ we show how to bound 
the lengths of the boundary geodesics of $P$ 
by induction on the depth in the spanning tree.
The root vertex of $\Gamma_n$ is the unique vertex of minimal depth 
and we begin by this case.

Suppose that $P$ is the root vertex
with boundary geodesics $\alpha_{0}, \alpha_1$ and $\beta_1$, where $\alpha_{0}$ corresponds to the boundary of the surface $\Sigma$.
Let  $\tau_0$ and $\tau'_0$ are as in Lemma \ref{ortho determines alpha} with endpoints on $\alpha_{0}$, 
where we suppose $\tau'_{0}$ going around $\alpha_{1}$.
Then  we  have an upper bound 
for $\ell(\tau_{0})$ namely, by combining Lemma \ref{cord bounds} and $x<\sinh x$ for $x>0$
$$\ell(\tau_{0})
\leq  2 \text{arcsinh} \left( \frac{2\exp(\ell(\alpha_1)/2)  }{ \ell(\alpha_{0} )} \right) 
\leq  2 \text{arcsinh} \left( \frac{2\exp(B/2) }{ \ell(\alpha_{0})} \right).
$$
To simplify notation we define a function
$$F(x):= 2 \text{arcsinh} \left( 2\exp(B/2)/x \right) + 2\ell(\alpha_{0}) $$ 
so that
\begin{eqnarray}
\ell(\tau_{0}) & \leq & F(\ell(\alpha_{0})) -  2\ell(\alpha_{0}) \\
\ell(\tau_{0}') & \leq & F(\ell(\alpha_{0}))
\end{eqnarray}
By Lemma \ref{ortho determines alpha}, we have 
$$ \cosh(\ell(\tau'_{0})/2) = 2\cosh(\ell(\alpha_1)/2)\cosh(\ell(\tau_{0})/2),$$
so that  $\ell(\tau_{0}') > \ell(\tau_{0})$ and moreover
$$  2\cosh( \ell(\alpha_{1})/2) 
= \frac{\cosh( \ell(\tau_0')/2)}{\cosh(  \ell(\tau_0)/2)},$$
The ortho spectrum  $\mathcal{O}(X)$
is discrete so there are only finitely many distinct values of the 
ortho spectrum  $\mathcal{O}(X)$
less than $F(\ell(\alpha_{0}))$.
Thus one can bound $\ell(\alpha_1)$ from below
using the hyperbolic granulosity
at  $F(\ell(\alpha_{0})) $.
Note that, by symmetry,
this lower bound is also a lower bound for 
$\ell(\beta_1)$.

%this doesn't work
%$$2\cosh(a/2) \geq \frac{\cosh((t + m)/2)}{\cosh(t/2)} 
%\geq 1 +  \frac{\cosh(m/2)}{\cosh(t/2)} 
%\geq 1 +  \cosh(m/2)
%\left (1 - \frac{4L}{\ell(\alpha_{0} ) + 4L } \right)
%$$

Now suppose that we have bounded $\alpha_n$ from below
and consider the pair of pants with boundary geodesics
$\alpha_n, \alpha_{n+1}, \beta_{n+1}$.
We now construct arcs
$\tau_n, \tau_n'$ (see Figure \ref{fig.tan_n}).
First join $\alpha_0$ and $\alpha_{n}$ by a shortest geodesic $a$, and then connect the endpoint of $a$ on $\alpha_{n}$ and $\alpha_{n+1}$ by a shortest geodesic $a'$. 
The concatenation $a\cup a'$ is a path connecting $\alpha_{0}$ and $\alpha_{n+1}$.
Then consider the pair of pants $P'_{n+1}$ whose boundary curves are $\alpha_{0}$, $\alpha_{n+1}$ and the boundary of a small neighborhood of 
$\alpha_{0}\cup\alpha_{n+1}\cup \{a\cup a '\}$.
Now we define $\tau_{n}$ and $\tau_{n+1}$ as in Lemma \ref{ortho determines alpha} applied for $P'_{n+1}$
so that 
\begin{equation}\label{eq.tau_{n}}
 2\cosh( \ell(\alpha_{n+1})/2) 
= \frac{\cosh( \ell(\tau_n')/2)}{\cosh(  \ell(\tau_n)/2)}.
\end{equation}
Let %$a$ be the shortest arc from  $\alpha_0$ to $\alpha_n$ and 
  $\tau, \tau'$ be the ortho geodesics in $P_n$ as before.
Then the arc $\tau_n$ (resp. $\tau'_{n}$) is the geodesic representative of a piecewise geodesic constructed
by following $a$ to $\alpha_n$,
traveling round $\alpha_n$ to the foot of $\tau$ (resp. $\tau'$),
following $\tau$ (resp. $\tau'$)
then going round $\alpha_n$ to the foot of $a$,
and finally returning to $\alpha_0$ via $a$.
There are choices of the direction when we go around $\alpha_{n}$, and we choose the direction so that corresponding geodesic representative
$\tau_{n}$ (resp. $\tau_{n}'$) sits on $P'_{n+1}$.
%Then Eq. \ref{eq.tau_{n}} follows Lemma \ref{ortho determines alpha}.

 \begin{figure}[H]
  \centering
  \subfloat  
{\includegraphics[scale=.23]{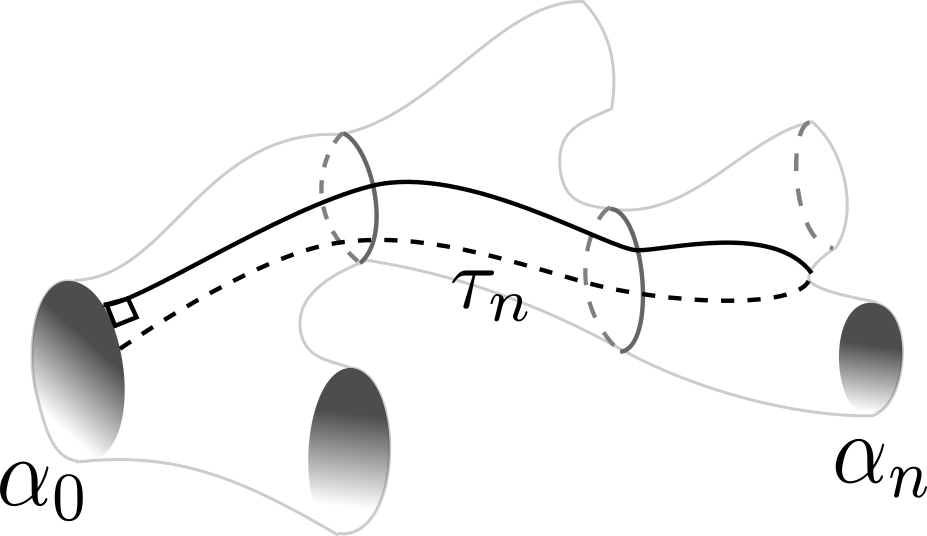} 
}
 \hfill
  \subfloat
  {\includegraphics[scale=.23]{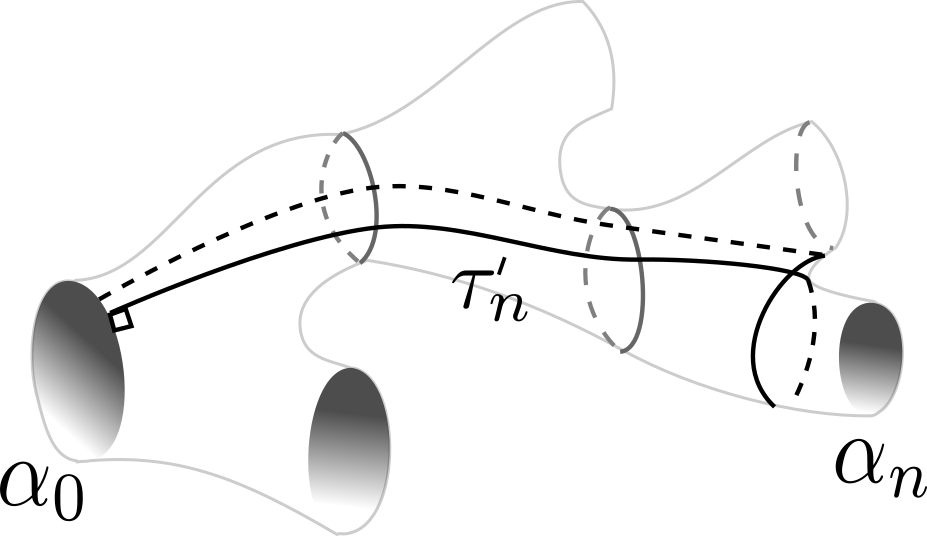} 
} 

%  \caption{Variation of the level sets as 
%  one of the boundary length increases (on the left) 
%  or decreases (on the right).
%   }
\caption{The ortho geodesics $\tau_n$ and $\tau_n'$.}
   \label{fig.tan_n}
\end{figure}
%
%and we bound their lengths 
%by constructing piecewise geodesic curves
%in the same homotopy classes rel $\alpha_{0}$. %$\partial \Sigma$.

Since $\tau_n$ is the shortest curve in the homotopy class rel $\partial \Sigma$
its length  is bounded above by 
the piecewise geodesic and  we obtain  
$$ \ell(\tau_n) \leq 2 \ell(a) + \ell(\tau) +  \ell(\alpha_n) 
\leq  2 \ell(a) + F(\ell(\alpha_n)).$$
Further, by constructing a  piecewise geodesic 
in the same homotopy class as $\tau'_n$, one sees 
$$ \ell(\tau'_n)  \leq 2 \ell(a) +  F(\ell(\alpha_n)) + B.$$
Once we bound $\ell(a)$ from above, the same argument as before shows there is a 
strictly positive lower bound for $\ell(\alpha_{n+1})$ 
which does not depend on $k$.

We now show how to get an  upper bound of $\ell(a)$.
Given a pair of pants with boundary $\alpha,\beta,\gamma$,
we have
\begin{equation}\label{eq.distance between boundary}
\cosh(d(\gamma,\alpha)) = \\
\sinh(\ell(\gamma))\sinh(\ell(\alpha))\cosh(\ell(\beta))-\cosh(\ell(\gamma))\cosh(\ell(\alpha)).
\end{equation}
Hence one obtains an upper bound from $d(\gamma,\alpha)$ 
from the bounds for the length of $\alpha,\beta$ and $\gamma$.
There is  a piecewise geodesic curve $\sigma$ homotopic to $a$
obtained by traveling alternately along the shortest ortho geodesic and 
the arcs of the boundary geodesic of the pants connecting 
the feet of the  ortho geodesics.
Then Eq.(\ref{eq.distance between boundary}) 
and upper bound of the length of the pants curve yield an upper bound for  
$\ell(\sigma)$ 
and hence for $\ell(a)$.

\end{proof}

\section{Abelian covers and the spectra} \label{sec.covers}

We now give a simple construction
which shows that the ortho spectrum
fails to distinguish surfaces with different lengths of systole
and so it fails to distinguish surfaces which are not  isometric.
To do this it is convenient to make a special choice 
for the   hyperbolic structure $X$ on $\MM$
which will make it easy to determine the length
of the systole on degree $n$ covers of $X$.

For completeness we include the following lemma
from which  the reader should be able to see 
why our construction is much simpler than Sunada's.

\begin{lem}
\label{degree d covers}
Let $\tilde{X}$ be an orientable hyperbolic surface with boundary:

\begin{enumerate}
\item
If $G$ is  group acting freely on $\tilde{X}$.
Then the subgroup of $G$ that leaves 
an ortho geodesic $\tilde{\tau}$  invariant is trivial.
\item
Let $\pi: \tilde{X} \rightarrow X$ be a regular degree $d$ 
cover of an orientable hyperbolic surface with boundary $X$.
If $\tau$ is an ortho geodesic on $X$
then it is covered by exactly $d$ ortho geodesics on $ \tilde{X}$.
\end{enumerate}
\end{lem}
\begin{proof}
The second point follows from the first by taking $G$ 
to be the group of deck transformations of the cover.

To prove the first point one considers
 $H$ the stabiliser of $\tilde{\tau}$  in $G$.
%Any $h\in H$ leaves the end points of $\tilde{\tau}$ invariant
%and, 
%since $X$ is orientable,
%if $h$ fixes them it must be the identity.
%If $h$ permutes the endpoints
%then it must fix the midpoint of $\tilde{\tau}$
%but this contradicts the fact that  $G$ acts freely.
Any $h\in H$ leaves the end points of $\tilde{\tau}$ invariant and fixes its midpoint,
which contradicts the fact that $G$ acts freely.
\end{proof}

\begin{lem}
\label{systole in covers}
For any $n>0$ there is a choice of  hyperbolic structure $X$ on $\MM$
such that:
\begin{enumerate}
\item 
The systole is shorter than
 $1/n$ times the length of the next shortest curve which has length at least $2\log(1+\sqrt 2)$.
 \item
Let $\tilde{X}$ 
be a degree $n$ cover of $X$
then the systole of the cover has 
length at most $n \ell(\alpha)$.
 \end{enumerate}
\end{lem}

%\begin{center}
%\includegraphics[scale=.5]{pants_generators.png} 
%\end{center}

\begin{proof}
Let $\alpha$ be  an essential  closed simple curve on $\MM$
and $\gamma$ a closed %non trivial 
curve homotopic to the boundary
$\partial \MM$.
By the Collar Lemma 
for any hyperbolic structure on $\MM$
such that $\ell(\alpha)$ is sufficiently small
the second shortest curve on $\MM$ 
is necessarily a curve $\alpha'$ disjoint 
from $\alpha$.
Note that $\alpha'$ is not simple since when we cut $\MM$ along $\alpha$ we have a pair of pants, and there is no simple closed curves other than the boundary on the pair of pants.
It is a result of Hempel \cite[Corollary 3.6]{Hempel} that every non-simple has length greater or equal to $2\log(1+\sqrt 2)$.

%Cut along $\alpha$ 
%to obtain a pair of pants $P$ 
%which inherits an incomplete hyperbolic metric from $\MM$.
%One can equip this  $P$ with a Poincar\'e metric which 
%is a complete hyperbolic metric 
%so that  equipped with this metric it  is isometric
%to the three punctured sphere.
%The identity  map on $P$ is a holomorphic map 
%so  by the Schwarz-Pick lemma
%the map from the  incomplete metric to the Poincar\'e metric 
%is a contraction.
%Since $\alpha'$ is disjoint from $\alpha$
%there is a closed geodesic on $P$
%of the same length.
%So,  since the identity is a contraction,
%$$\ell(\alpha') \geq \ell_\infty(\alpha') $$
%where $\ell_\infty(\alpha')$ 
%denotes the length of the geodesic homotopic to $\alpha'$
%for the Poincar\'e metric on $P$.
%Now on the three punctured sphere
%the length of the shortest closed geodesic is
%$2 \,2\mathrm{arcsin(1)}$ so that
%$$\ell(\alpha') \geq  2 \, 2\mathrm{arcsin(1)}.$$

Now choose a hyperbolic metric on $P$ such that
there is a boundary  geodesic of length
$\ell(\gamma)$
and the other two boundary geodesics are of length 
$$(1/n) \times 2\log(1+\sqrt 2).$$
By identifying these two boundary geodesics 
one obtains a hyperbolic structure $X$ on $\MM$.
Then $\ell(\alpha)$ on $X$ equals $2\log(1+\sqrt 2)/n$ and 
the length of $\ell(\alpha')\geq 2\log(1+\sqrt 2) = n\ell(\alpha')$.

%If  $\alpha$ is an essential  closed simple curve 
%then $\ell(\alpha) < \ell(\alpha')$ 
%where $\alpha'$ any  figure eight curve on $\MM$ 
%disjoint from $\alpha$.
%To see this cut along $\alpha$ 
%to obtain a pair of pants $P$.
%We claim that on $P$
%$$\ell(\gamma')  > \ell(\gamma)$$
%where $\gamma$ is any  choice  of the three boundary geodesics of $P$
%and $\gamma'$ any choice of  figure eight curve on $P$.
%We chose generators 
%$a,b$ (see Figure \ref{}) for the fundamental group of $P$
%such that $ab$ is homotopic to a boundary geodesic
%and $ab^{-1}$ to the figure eight curve.
%We identify with a subgroup $\Gamma$ of $\mathrm{SL}(2, \RR)$ 
%which uniformizes the hyperbolic structure on $P$.
%By the Negative Trace Theorem (see Maskit)  
%we can choose $\Gamma$
%such that 
%$$\tr a >2,\,  \tr b > 2,\, \tr ab <  -2.$$
%Using the trace relations we have $\tr a \, \tr b  = \tr ab + \tr ab^{-1}$
%so that 
%$$\tr ab^{-1} = \tr a \, tr b  + \tr ab > \tr a \, tr b >   2\, \tr a.$$
%In terms of geodesic lengths this  inequality is equivalent to
%$$2 \cosh(\ell(\gamma')/2) > 4 \cosh(\ell(\gamma')/2)$$
%where $\gamma$ is any  choice  of the three boundary geodesics of $P$
%and $\gamma'$ any choice of  figure eight curve on $P$.
% 
%By the Collar Lemma there is $\epsilon >0$ 

\end{proof}

\begin{thm}\label{covers}
Let $k>0$ and  set $n = 2^k$
and  $X$ be a hyperbolic structure on $\MM$ obtained as in
Lemma \ref{systole in covers}.
Then there are $k$ cyclic covers $\tilde{X}_i$ of $X$ 
such that the length of the systole
is different for each of these covers.

\end{thm}

\begin{proof}
Let $\alpha$ denote the shortest essential  closed geodesic 
on $X$.
Since the fundamental group of $\MM$ 
is free on two generators there is a homomorphism  onto $\ZZ$.
We choose generators $\alpha, \beta$ for $
 \pi_1(\MM)$ such that $\alpha$ is freely homotopic to $\alpha$
and define a homomorphisms by
$$\pi_1(\MM) \rightarrow \ZZ,\,  
\alpha \mapsto 2^m,\, \beta \mapsto 1.$$
Reducing modulo $2^k$ the image of $\alpha$ 
we obtain a surjective homomorphism 
$$\pi_1(\MM) \rightarrow \ZZ/2^k \ZZ$$ 
and we denote by $\tilde{X}_m$
the   regular degree $2^k$   covering   corresponding to the kernel 
of this homomorphism.

Each closed geodesic 
$\tilde{\beta'} \subset \tilde{X}_m$
covers some closed geodesic 
$\beta' \subset X$ 
and the degree of this cover
is the order of the image of 
$\beta'$ in $\ZZ/2^k \ZZ$.
In particular $\ell( \tilde{\beta'}) \geq \ell( \beta' )$
and if $\beta' \neq \alpha$
$$\ell( \tilde{\beta'}) \geq 2 \, \log(1+\sqrt 2) >  2^k \ell(\alpha),$$
so, 
since $2^k$ is the degree of the cover 
$\tilde{X}_m \rightarrow X$,
the shortest closed geodesic on 
$ \tilde{X}_m$ must cover $\alpha$.

Finally, we can compute the length of the systole: if  $\tilde{\alpha} \subset \tilde{X}_m$
covers  $\alpha $ then
$$ \ell(\tilde{\alpha}) =  2^{k-m} \times \ell(\alpha),$$
so that the systole distinguishes the covers $\tilde{X}_m$.

\end{proof}

\appendix
\section{dimension and ortho spectrum in dimension 3}

Let $M$ be a compact hyperbolic $3$-manifold with totally geodesic boundary.
We let $\Gamma$ be a Kleinian group such that the convex core of $\mathbb{H}^{3}/\Gamma$ is isometric to $M$.
It is observed by several authors \cite{BK, Cal} that the ortho spectrum of $M$ determines the volume of $M$ and the area of $\partial M$.
In this section, we prove that the Hausdorff dimension of the limit set $\Lambda(\Gamma)$ of $\Gamma$ can also be determined by the 
ortho spectrum of $M$.
Since the boundary of $M$ is totally geodesic and hence the limit sets of the boundary subgroups are round circles,
the Ahlfors' finiteness theorem \cite{Ahl} implies that the limit set $\Lambda(\Gamma)$ is a circle packing of $\partial\mathbb{H}^{3} = \mathbb{C}\cup \{\infty\}$.
By considering conjugate if necessary, we may assume that $\infty\in\partial\mathbb{H}^{3}\setminus\Lambda(\Gamma)$.
First, we recall the work of Parker.
Let $\mathcal{C}$ be a circle packing of $\mathbb{C}=\partial\mathbb{H}^{3}\setminus\{\infty\}$.
When we discuss the {\em radius} of circles in $\mathcal{C}$, we consider the Euclidean metric on $\mathbb{C}$.
Let $\mathcal{R}(\mathcal{C})$ denote the set of radii of the circles in $\mathcal{C}$ counted with multiplicity. 
Then {\em the circle packing exponent} of $\mathcal{C}$ is 
$$e_{\mathcal{C}}:=\inf\{t:\sum_{r\in \mathcal{R}(\mathcal{C})} r^{t} <\infty\}.$$
Parker showed the following.
\begin{thm}[\cite{Par}]\label{thm.Parker}
Let $\Gamma$ be as above. Suppose $\infty\in\partial\mathbb{H}^{3}\setminus \Lambda(\Gamma)$.
Then the circle packing exponent $e$ of $\Lambda(\Gamma)$ equals the Hausdorff dimension $d$ of $\Lambda(\Gamma)$.
\end{thm}

We now define the exponent of ortho spectrum as follows.
Let $\mathcal{O}(M)$ denote the ortho spectrum of $M$.
Then we define the ortho spectrum exponent by
$$e_{\mathcal{O}}:=\inf\{t:\sum_{l\in \mathcal{O}(M)}\frac{1}{e^{tl}}<\infty\}.$$
Our goal is the following.
\begin{thm}
Let $M$ and $\Gamma$ be as above. Suppose $\infty\in\partial\mathbb{H}^{3}\setminus \Lambda(\Gamma)$.
Then the ortho spectrum exponent $e_{\mathcal{O}}$ of $M$ equals the Hausdorff dimension $d$ of $\Lambda(\Gamma)$.
\end{thm}
\proof
Let $\gamma$ be a circle in $\Lambda(\Gamma)$ and $H<\Gamma$ be the (setwise) stabiliser of $\gamma$.
By applying conjugate if necessary, we may suppose that $\gamma$ is the unit circle in $\mathbb{C}$.
As $H$ acts on the unit disc surrounded by $\gamma$ discontinuously, 
we have a fundamental domain $D$ which contains $0$.
First we claim that there is a constant $C>0$ which depends only on $M,H$ and $D$ with the following property.
Let $\delta$ be a circle in $\Lambda(\Gamma)$ meeting $D$, we let $l$ denote the length of the ortho geodesic connecting $\delta$ and $\gamma$.
Let us show that the radius $r$ of $\delta$ satisfies 
\begin{eqnarray}\label{eq.ortho v.s. radius}
\frac{1}{Ce^{l}}\leq r \leq \frac{1}{e^{l}}.
\end{eqnarray}
The ortho geodesic connecting $\delta$ and $\gamma$ lies on a totally geodesic hyperplane intersecting with $\delta$ and $\gamma$.
Hence it suffices to discuss the case of dimension two, and we now suppose $\gamma=\{-1,1\}\subset\partial\mathbb{H}^{2}$.
The convex hull $C(\gamma)\subset\mathbb{H}^{2}$ of $\gamma$ is the half circle of radius $1$ centred at $0$.
Let $\delta'\subset\mathbb{H}^{2}$ denote the half circle centred at $0$ with hyperbolic distance $l$ from $C(\gamma)$.
The radius $r'$ of $\delta'$ equals $1/e^{l}$.
Since the $l$-neighbourhood (in hyperbolic metric) of the geodesic $\gamma$ becomes thickest (in terms of Euclidean metric) 
at the top, among geodesics of distance $l$ apart from $\gamma$,
$\delta'$ has the largest possible radius.
Hence we have $r\leq 1/e^{l}$.
We now consider an isometry 
\(
  A = \left(
    \begin{array}{rr}
      s & t \\
      u & v 
    \end{array}
  \right)
 \in \mathrm{PSL}(2,\mathbb{R})
\)
which preserves $C(\gamma)$ and maps $\delta'$ to $\delta$.
Without loss of generality, we may suppose $s>0$.
Since $A$ preserves $-1$ and $1$, we have $s = v$ and $t=u$, and hence, $s^{2}-t^{2} = 1$.
Then the centre of $\delta$ is $t/s$.
Note that since $M$ is compact and there is a lower bound of the length of ortho geodesics,
 there exists $1>\epsilon>0$ which depends only on $M, H$ and $D$ such that $\partial \delta\subset(-1+\epsilon,1-\epsilon)$
whenever $\delta$ meets $D$.
A simple calculation shows that the radius $r$ of $\delta$ equals $$\frac{r'}{s(tr'+s)}.$$
Since $t/s\leq 1-\epsilon$ and $s^{2}-t^{2} = 1$, we have $t\leq (1-\epsilon)s$, and $$s^{2}\leq \frac{1}{1-(1-\epsilon)^{2}}.$$
Therefore as $r'\leq 1-\epsilon$, and $r' = 1/e^{l}$, we have
$$r = \frac{1/e^{l}}{s(tr'+s)}\geq \frac{1/e^{l}}{s^{2}((1-\epsilon)r'+1)}\geq \frac{1-(1-\epsilon)^{2}}{1+(1-\epsilon)^{2}}\cdot \frac{1}{e^{l}},$$
and the claim follows.

Now we prove $e_{\mathcal{O}}\leq d$.
By Theorem \ref{thm.Parker}, we have $e_{\Lambda(\Gamma)} = d$.
For a subset $T\subset\partial \mathbb{H}^{3}$, 
let $$e_{T}:=\inf\{t:\sum_\text{$r$: radius of circles meeting $T$} r^{t}<\infty\}.$$
Then we have $e_{T}\leq e_{\Lambda(\Gamma)} = d$ for any $T$, in particular when $T=D$.
Let $B\subset \partial M$ be the component corresponding to $H$, and 
$\mathcal{O}_B$ denote the set of lengths of ortho geodesics of $M$ one of whose endpoints lies on $B$.
Then each element in $\mathcal{O}_B$ corresponds to an ortho geodesic which has a lift in $\mathbb{H}^3$ so that 
it connects $\gamma$ and a circle meeting $D$.
As there are only finitely many components in $\partial M$, we may suppose 
$$e_{B}:=\inf\{t:\sum_{l\in \mathcal{O}_B}\frac{1}{e^{tl}}<\infty\}$$
equals $e_\mathcal{O}$.
Since the constant multiplicable error in inequality (\ref{eq.ortho v.s. radius}) does not change the critical exponents,
we have $e_B = e_{D}$ and hence $e_D = e_\mathcal{O}$.
Hence we have $e_{\mathcal{O}}\leq d$.

For the converse inequality, we use the work of Larman \cite{Lar} which says that for any compact subset $K\subset\mathbb{C}$
and a circle packing $\mathcal{C}$ of $K$, 
the circle packing exponent of $\mathcal{C}$ is greater than or equal to the Hausdorff dimension of $\mathcal{C}$.
We first take large enough ball $K'$ centred at $0$ and contained in $\gamma$.
Then we let $K$ to be the subspace obtained by removing all circles intersecting with $\partial K'$ from $K'$.
By taking $K'$ large enough we may suppose that $K$ contains all the circles meeting $D$.
As $D$ is a fundamental domain and $K$ is compact, we may cover $K$ by finitely many translates $\gamma_i D$ of $D$ 
where $\gamma_i\in\Gamma$.
By the same argument as above we have $e_{\gamma_i D} = e_\mathcal{O}$ for every $i$, which implies 
$$e_{\cup_i \gamma_i D} = e_\mathcal{O}.$$
Hence we have $e_\mathcal{O} = e_{D} = e_{\cup_i \gamma_i D} = e_{K}$, 
and the Hausdorff dimension of the restriction of $\Lambda(\Gamma)$ on $K$ equals $d$.
Therefore we have $e_{\mathcal{O}} = e_{K}\geq d$ by Larman.
Thus $e_{\mathcal{O}} = d$ follows.

%
%By the definition of the Hausdorff dimension, the Hausdorff dimension of $D$ coincides with $d$.
%
%
%
%i) it suffices to see only at fundamental domain of $H$.
%ii) on the fundamental domain we have $r\sim 1/e^{l}$.
%iii) critical exponent of the circles in fund. domain $\leq d$ by Parker
%iv) critical exponent of the circles in fund. domain $\geq d$ by Larman, here we use the fact that hdf dim of $\Lambda(\Gamma)$ equals hdm dim of $\Lambda(\Gamma)\cap\text{ fund. domain}$

\qed

\end{document}